\UseRawInputEncoding
\documentclass[12pt]{amsart}

\usepackage[all]{xy}
\usepackage{fullpage}
\usepackage{latexsym}
\usepackage{amsmath}
\usepackage{amsfonts}
\usepackage{amssymb}
\usepackage{amsthm}
\usepackage{eucal}
\usepackage{enumerate,yfonts}
\usepackage{mathrsfs}
\usepackage{graphicx}
\usepackage{graphics}
\usepackage{epstopdf}
\usepackage{amscd}
\usepackage{bbm}
\usepackage{hyperref}
\usepackage{url}
\usepackage{color}
\usepackage{bbm}
\usepackage{cancel}
\usepackage{enumerate}
\usepackage{amsmath,amsthm}
\usepackage{amssymb}
\usepackage{epsfig}
\usepackage{pstricks}
\usepackage{xy}
\usepackage{xypic}
\usepackage{tikz}
\usepackage{tikz-cd}

\newtheorem{thm}{Theorem}[section]
\newtheorem{corollary}[thm]{Corollary}
\newtheorem{lemma}[thm]{Lemma}

\newtheorem{thm-dfn}[thm]{Theorem-Definition}

\theoremstyle{definition}
\newtheorem{definition}[thm]{Definition}

\numberwithin{equation}{section}

\theoremstyle{remark}
\newtheorem{remark}{Remark}[section]

\newcommand{\fg}{{\mathfrak g}}
\newcommand{\ft}{{\mathfrak t}}
\newcommand{\fl}{{\mathfrak l}}
\newcommand{\fb}{{\mathfrak b}}

\newcommand{\fp}{{\mathfrak p}}

\newcommand{\fc}{{\mathfrak c}}
\newcommand{\fk}{{\mathfrak k}}

\newcommand{\rW}{{\mathrm W}}

\newcommand{\bC}{{\mathbb C}}

\newcommand{\bG}{{\mathbb G}}
\newcommand{\bZ}{{\mathbb Z}}

\newcommand{\mF}{\mathcal{F}}

\newcommand{\mO}{\mathcal{O}}

\newcommand{\on}{\operatorname}

\newcommand{\is}{\simeq}

\newcommand{\Loc}{\on{Loc}}

\newcommand{\quash}[1]{}  
\newcommand{\nc}{\newcommand}

\newcommand{\fraks}{{\mathfrak s}}

\newcommand{\bbH}{{\mathbb H}}

\newcommand{\bbO}{{\mathbb O}}

\newcommand{\bbR}{{\mathbb R}}

\newcommand{\calC}{{\mathcal C}}

\newcommand{\calF}{{\mathcal F}}

\nc{\so}{\mathfrak{so}}
\nc{\al}{{\alpha}} \nc{\be}{{\beta}} \nc{\ga}{{\gamma}}
\nc{\ve}{{\varepsilon}} \nc{\Ga}{{\Gamma}} 
\nc{\La}{{\Lambda}}
\nc{\ho}{\text{hol}}

\def\so{{\mathfrak{so}}}

\nc{\ad }{{\on{ad }}}

\nc{\aff}{{\on{aff}}} \nc{\Aff}{{\mathbf{Aff}}}

\nc{\der}{{\on{der}}}

\nc{\diag}{{\on{diag}}}

\nc{\Fl}{{\calF\ell}}

\newcommand{\Gr}{{\on{Gr}}}
\nc{\Hg}{{\on{Higgs}}}
\newcommand{\Hom}{{\on{Hom}}}

\nc{\Id}{{\on{Id}}}

\nc{\Ind}{{\on{Ind}}}

\nc{\Op}{{\on{Op}}}

\newcommand{\pr}{{\on{pr}}}

\nc{\res}{{\on{res}}}

\nc{\tr}{{\on{tr}}}

\nc{\GSp}{{\on{GSp}}} \nc{\GU}{{\on{GU}}} \nc{\SL}{{\on{SL}}}
\nc{\SU}{{\on{SU}}} \nc{\SO}{{\on{SO}}}

\nc{\nh}{{\Loc_{J^p}(\tau')}}
\nc{\bnh}{{\Loc_{\breve J^p}(\tau')}}

\nc{\bU}{{\overline{U}}} \nc{\IC}{{\on{IC}}}

\newcommand{\beqn}{\begin{equation*}}
\newcommand{\eeqn}{\end{equation*}}

\newcommand{\beq}{\begin{equation}}
\newcommand{\eeq}{\end{equation}}

\tikzset{node distance=2em, ch/.style={circle,draw,on chain,inner sep=2pt},chj/.style={ch,join},every path/.style={shorten >=4pt,shorten <=4pt},line width=1pt,baseline=-1ex}

\quash{
\setlength{\parskip}{1ex}
\setlength{\oddsidemargin}{0in}
\setlength{\evensidemargin}{0in}
\setlength{\textwidth}{6.5in}
\setlength{\topmargin}{-0.15in}
\setlength{\textheight}{8.6in}
\topmargin-0.5cm \textheight22cm \oddsidemargin1.2cm \textwidth14cm}
\begin{document}
\title{Lorentzian and Octonionic Satake equivalence}

         \author{Tsao-Hsien Chen and John O'Brien }
        \address{}
         \email{}
\thanks{}
\thanks{}

\maketitle   
\begin{abstract}
We establish a derived geometric Satake equivalence for the 
real group $G_\bbR=PSO(2n-1,1)$ (resp. $PE_6(F_4)$), to be called the Lorentzian Satake equivalence (resp.  Octonionic Satake equivalence).
By applying the real-symmetric correspondence for affine Grassmannians, we obtain a derived geometric Satake equivalence for the splitting rank
symmetric variety $X=PSO_{2n}/SO_{2n-1}$
(resp. $PE_6/F_4$).
As an application, we compute the stalks of the $\on{IC}$-complexes for spherical orbit closures in 
the real affine Grassmannian for $G_\bbR$
and the loop space of $X$. We show the stalks are given by the Kostka-Foulkes polynomials
for $GL_2$ (resp. $GL_3$) but 
with $q$ replaced by $q^{n-1}$ (resp.  $q^4$).
\end{abstract} 

\section{Introduction}
This is a companion paper of  joint work with Nadler and Macerato \cite{CMNO} on 
the derived geometric Satake equivalence for the real quaternionic linear groups.
We consider the real form 
$G_\bbR=PSO(2n-1,1)$ (resp. $PE_6(F_4)$) of
the  projective special orthogonal group 
$G=PSO_{2n}=SO_{2n}/\mu_2$ (resp. the adjoint  exceptional group $PE_6$).
Let $X=PSO_{2n}/SO_{2n-1}$ (resp.  $PE_6/F_4$) be the 
corresponding symmetric variety of splitting rank.
Let $G^\vee=Spin_{2n}$ (resp.  $E_6$) be the dual group of $G$
and let $G^\vee_X=SL_2$ (resp. $SL_3$) be the dual group of 
$G_\bbR$ in \cite{GN,N}. 
Let $L^\wedge_X=Spin_{2n-3}$ (resp. $G_2$) be the subgroup of $G^\vee$
introduced in \cite{KS}.
Denote by $\fg_X^\vee$ and  $\fl_X^\wedge$ the Lie algebra of $G^\vee_X$ and $L_X^\wedge$.
Denote by $n_X=2n-2$ (resp. $8$)

Let $\Gr_{\bbR}$  be the real affine Grassmannian for $G_\bbR$
and let $D_{L^+G_\bbR}(\Gr_{G_\bbR})$ be the 
real Satake category, that is, the dg-category of 
$L^+G_\bbR$-equivariant constructible complexes on $\Gr_{\bbR}$.
Let $D^{G_X^\vee}_{}(\mO(\fg_X^\vee[n_X]\times\fl_X^\wedge[2]//L^\wedge_X))$
be the dg-category of $G_X^\vee$-equivariant dg-modules 
over the dg-algebra $\mO(\fg_X^\vee[n_X]\times\fl_X^\wedge[2]//L^\wedge_X)$ 
of ring of functions on the graded scheme $\fg_X^\vee[n_X]\times\fl_X^\wedge[2]//L^\wedge_X$ (equipped with trivial differential).
Here $G^\vee_X$ acts on $\fg_X^\vee[n_X]$ via
the adjoint action and 
$\fl_X^\wedge[2]//L^\wedge_X$ is the adjoint quotient 
of $\fl_X^\wedge[2]$ equipped with the trivial $G_X^\vee$-action.
The main result of the paper is the following version of 
derived geometric Satake equivalence for $G_\bbR$:
\begin{thm}[Theorem \ref{main}]\label{main intro}
There is an equivalence of dg-categories
\[D_{L^+G_\bbR}(\Gr_{\bbR})\is D^{G_X^\vee}_{}(\mO(\fg_X^\vee[n_X]\times\fl_X^\wedge[2]//L^\wedge_X))\]
\end{thm}
The proof relies on the techniques developed in \cite{CMNO},
the formality property of real Satake category in \cite{CN}, and the computation 
of equivariant cohomology of $\Gr_\bbR$ in \cite{O}. 
The key steps are the 
parity vanishing results in Section \ref{Parity vanishing} and the computation of the
dg-extension algebra in Theorem \ref{computation of Ext}.

Let $D_{L^+G}(LX)$ be the dg-category of 
$L^+G$-equivaraint constructible complexes on the loop space 
$LX$ of $X$. There is a real-symmetric equivalence 
$D_{L^+G}(LX)\is D_{L^+G_\bbR}(\Gr_{\bbR})$ in \cite{CN}.
As a corollary of Theorem \ref{main}, we obtain the following version 
of derived geometric Satake equivalence for $X$ establishing instances of
relative Langlands duality in \cite{BZSV}:
\begin{thm}\label{main2 intro}
There is an equivalence of dg-categories
\[D_{L^+G}(LX)\is D^{G_X^\vee}_{}(\mO(\fg_X^\vee[n_X]\times\fl_X^\wedge[2]//L^\wedge_X))\]
\end{thm}

\begin{remark}
In the case $X=PSO_{2n}/SO_{2n-1}$, the theorem above is proved in \cite[Theorem 4.6.1]{S} by a different method.
\end{remark}

Similar to \cite[Section 3.7]{CMNO}, we have the following 
application on computing the $\IC$-stalks for
 spherical orbits closures in $\Gr_\bbR$ and $LX$.
 Recall that  the spherical orbits in $\Gr_\bbR$ and $LX$
 are parametrized by the set 
$\Lambda_S^+$ be the real dominant weight.
For any $\lambda\in\Lambda_S^+$ let 
$\Gr_\bbR^\lambda$  and $L^\lambda X$ be the corresponding 
spherical orbits and let 
$\IC^\bbR_\lambda$ and $\IC_\lambda^X$ be the $\IC$-complex for their closure respectively.
Let $K_{\lambda,\mu}(q)$ is the Kostka-Foulkes polynomials
for $GL_2$ (resp. $GL_3$).
\begin{thm}[Corollary \ref{n_X vanishing} and Theorem \ref{Kostka-Foulkes}]\label{Kostka-Foulkes intro}
Let $\lambda,\mu\in\Lambda_S^+$. For any $z\in\Gr_\bbR^\mu$ and $y\in L^\mu X$
we have 
\begin{enumerate}
\item
$\mathscr H^{i-\frac{1}{2}\dim_\bbR\Gr_\bbR^\lambda}(\IC^\bbR_\lambda)=\mathscr H^{i-\frac{1}{2}\dim_\bbR\Gr_\bbR^\lambda}(\IC^X_\lambda)=0\ \ 
\text{for}\ \ \  n_X\nmid i$
\item
$K_{\lambda,\mu}(q)=\sum_{i}\mathscr H_z^{-n_Xi-\frac{1}{2}\dim_\bbR\Gr_\bbR^\lambda}(\IC^\bbR_\lambda)q^i=\sum_{i}\mathscr H_y^{-n_Xi-\frac{1}{2}\dim_\bbR\Gr_\bbR^\lambda}(\IC^X_\lambda)q^i$
\end{enumerate}

\end{thm}

\subsection{Acknowledgements}

We would like to thank David Nadler and Mark Macerato
 for useful discussions.
 T.-H. Chen also thanks the NCTS-National Center for Theoretical Sciences at Taipei where parts of this work were done.
 The research of T.-H. Chen is
supported by NSF grant DMS-2143722.

\section{Notation related to real groups}

\subsection{Real groups}
Let $G_\bbR$ be a real form of a complex connected reductive group $G$
and let $K_\bbR\subset G_\bbR$ be a  maximal compact subgroup.
Denote by $\fg_\bbR$ and $\fk_\bbR$ their Lie algebras.
 We have the Cartan decomposition $\fg_\bbR=\fk_\bbR\oplus\fp_\bbR$.
Choose a maxmial abelian subspace $\fraks_\bbR\subset\fp_\bbR$
and let  $S_\bbR=\on{exp}(\fraks_\bbR)\subset G_\bbR$ 
be the corresponding maximal (connected) split torus.
We choose a minimal parabolic subgroup $P_\bbR$ with the Levi subgroup 
$L_\bbR=Z_{G_\bbR}(S_\bbR)$. 
We have $L_\bbR\is M_\bbR\times S_\bbR$ where $M_\bbR=Z_{K_\bbR}(S_\bbR)$. 
We choose a maximal torus $H_\bbR\subset L_\bbR$ 
and let $T_\bbR=K_\bbR\cap H_\bbR$ which is a maximal torus for both $K_\bbR$ and $M_\bbR$. 
We have $H_\bbR\is T_\bbR\times S_\bbR$.

We write $\frak m_\bbR,\ft_\bbR, \frak h_\bbR,$ etc, for the Lie algebras of 
$M_\bbR, T_\bbR, H_\bbR,$ etc, and $\frak m,\ft,\frak h,$ etc, for their complexification.
We write $K,S, P, M,$ etc, for the complexification of $K_\bbR, S_\bbR, P_\bbR, M_\bbR,$ etc, We fix a Borel subgroup $B\subset G$ such that $H\subset B\subset P$.

Let $W_X=N_K(\frak s)/Z_K(\frak s)$ be the little Weyl group of $G_\bbR$,
$W_K=N_{K^0}(T)/T$ be the Weyl group of $K$,
and $W_M=N_{M^0}(T)/T$ the Weyl group of $M$, where 
$K^0$ and $M^0$ are the neutral of $K$ and $M$.

Let $\Lambda_H=\Hom(\bC^\times,H)$ be the set of coweights of $H$ and 
$\Lambda_H^+$ be the set of dominant coweights with respect to $B$.
Let $2\rho_G$ be the sum of positive roots of $G$.
Let $\Lambda_S=\Hom(\bG_m,S)$ be the set of real coweights
and $\Lambda_S^+=\Lambda_S\cap\Lambda_H^+$ be the set of dominant real weights.

\subsection{Real affine Grassmannian}
Consider the 
real affine grassmannian $\Gr_\bbR=LG_\bbR/L^+G_\bbR$,
where $LG_\bbR=
G_\bbR(\bbR((t)))$ and $L^+G_\bbR=G_\bbR(\bbR[[t]])$
are the real loop group and arc groups.
For any $\lambda\in\Lambda_S^+$ let 
$\Gr_\bbR^\lambda=L^+G_\bbR\cdot t^\lambda$ be the corresponding real spherical orbit 
in $\Gr_\bbR$, see \cite[Proposition 3.6.1]{N}.

\subsection{Dual groups}
Let $G^\vee$ be the complex dual group of $G$.
In \cite{GN,N}, the authors associated to each real form $G_\bbR$
a canonical connected reductive subgroup $G_X^\vee\subset G^\vee$
called the dual group of $G_\bbR$.
We denote by $2\check\rho_{X}$  the sum of positive coroots of $G^\vee_X$.

Let $L^\vee\subset G^\vee$ be the dual Levi subgroup.
According to  \cite[Section 9]{KS}, the little Weyl group $W_X$
acts on $L^\vee$ and  $G^\vee_X$ is centralized by 
a finie index subgroup $L^\wedge_X$ of the fixed point subgroup $(L^\vee)^{W_X}$.
We denote by $\fg^\vee, \fg_X^\vee, \fl_X^\wedge$ the Lie algebra of $G^\vee, G^\vee_X, L_X^\wedge$.

\subsection{The real groups $PSO(2n-1,1)$ and $PE_6(F_4)$}\label{E6}
For the rest of the paper we focus on the case 
$G=PSO_{2n}$ or $PE_6$ (the adjoin from of $E_6$) and $G_\bbR$ is the 
Lorentz group 
$PSO(2n-1,1)$ or $PE_6(F_4)$.

In the case $G_\bbR=PSO(2n-1,1)$, we have 
$K=SO_{2n-1}$,
$G^\vee=Spin_{2n}$, $G^\vee_X=SL_2$,  $M=SO_{2n-2}$, $\Lambda_S\is\mathbb Z$,
$\Lambda^+_S\is\mathbb Z_{\geq0}$,
and $L_X^\vee=Spin_{2n-3}$

In the case $G_\bbR=PE_6(F_4)$, we have $K=F_4$, 
$G^\vee=E_6$ (the simply connected form of $E_6$), $G^\vee_X=SL_3$,  $M=Spin_8$,
$\Lambda_S\is\bZ\times\bZ$, $\Lambda^+_S\is\{(a,b)\in\bZ\times\bZ|a\geq b\}$,
and
$L_X^\vee=G_2$.

We denote by $\ft_X^\vee\subset\fg_X^\vee$ the Cartan subalgebra of 
diagonal matrices and  
$\fb_X^\vee\subset\fg_X^\vee$  the standard Borel subalgebra of 
upper triangular matrices. 
The Weyl group  of $G_X^\vee$ is isomorphic to the little Weyl group $W_X$
and we have $W_K\is W_M\rtimes W_X$.

We have a natural identification of $\Lambda_S$ with the weight lattice of $G^\vee_X$
and for any $\lambda\in\Lambda_S$ we have 
\beq\label{equ of paring}
\langle\lambda,\rho_G\rangle=n_X\langle\lambda,\check\rho_{X}\rangle
\eeq
where 
$n_X=2n-2$ or $8$
in the case $G_\bbR=PSO(2n-1,1)$ or $PE_6(F_4)$.

\section{Equivariant cohomology}
We first recall 
the following description of 
$M_\bbR$ and $K_\bbR$-equivariant cohomology of a point.

\begin{lemma}\label{equ of points}
(1) There are isomorphisms of graded scheme
\[\on{Spec}(H^*_{T_\bbR}(\on{pt}))\is\ft[2]\ \ \ \ \on{Spec}(H^*_{M_\bbR}(\on{pt}))\is\ft[2]//W_M\]
\[\on{Spec}(H^*_{K_\bbR}(\on{pt}))\is\on{Spec}(H^*_{M_\bbR}(\on{pt}))^{W_X}
\is(\ft[2]//W_M)^{W_X}\is
\ft[2]//W_K\]
(2) There are isomorphism of graded schemes
\[
\on{Spec}(H^*_{M_\bbR}(\on{pt}))\is\ft_X^\vee[n_X]\times\fl_X^\wedge[2]//L_X^\wedge
\]
\[\on{Spec}(H^*_{K_\bbR}(\on{pt}))\is\ft_X^\vee[n_X]//W_X\times\fl_X^\wedge[2]//L_X^\wedge.\]
\end{lemma}
\begin{proof}
Part (1) is standard and part (2) 
follows from comparing gradings on both sides
(in view of the description of $G_X^\vee$ and  $L_X^\wedge$ in Section \ref{E6}).

\end{proof}

We recall the following description of 
$T_\bbR$-equivariant homology of $\Gr_\bbR$
the  in \cite{O}.
Let $\mathbf e\in\fb_X^\vee$ the nilpotent Jordan matrix of maximal size (i.e., the 
standard regular nilpotent element in $\fg_X^\vee$).
Lemma \ref{equ of points} (2) induces 
a  map 
\beq\label{projection}
p^T:\ft[2]\to\ft[2]//\rW_M\is\ft_X^\vee[n_X]\times\fl_X^\vee[n_X]//L_X^\vee\stackrel{\pr}\to
\ft_X^\vee[n_X]
\eeq
Regard $\mathbf e$ as a constant map
$\mathbf e:\ft\to \mathbf e\subset\fb_X^\vee$
and consider the the map 
\beq\label{e^T}
e^T=p^T+\mathbf e:\ft\to \fb_X^\vee.
\eeq
Consider the constant group scheme $G_X^\vee\times\ft\to\ft$ over $\ft$
and let 
$(G_X^\vee\times\ft)^{e^T}\subset G_X^\vee\times\ft$ be the sub group scheme of stabilizers of the section $e^T$:
\[\xymatrix{(G_X^\vee\times\ft)^{e^T}\ar[r]\ar[d]&G_X^\vee\times\ft\ar[d]^{act}\\
\ft\ar[r]^{e^T\times\on{id}}&\fg_X^\vee\times\ft}\] 
where $act(g,t)=(ge^T(t)g^{-1},t)$.

\begin{thm}\cite[Theorem 8]{O}\label{equ}
There is an isomorphism of group schemes
\[
\on{Spec}(H_*^{T_\bbR}(\Gr_\bbR))\is(G_X^\vee\times\ft)^{e^T}
\]
\end{thm}

\section{Parity vanishing and fully faithfulness}\label{Parity vanishing}
Let $\omega\in\Lambda_S^+$ be the minuscule real coweight
and $\Gr_\bbR^\omega$ be the minuscule real spherical orbit.
We have 
here $\Gr_\bbR^\omega\is\mathrm S^{2n-2}$ is the $(2n-2)$-sphere if $G_\bbR=PSO(2n-1,1)$
and $\Gr_\bbR^\omega\is\bbO P_2$ is the octonionic projective plane if $G_\bbR=PE_6(F_4)$.
Note that $\dim_\bbR\Gr_\bbR^{\omega}=n_X$.

\begin{lemma}\label{paving}
Consider the $k$-fold convolution morphism
\[m_k:\Gr_\bbR^\omega\tilde\times\cdot\cdot\cdot\tilde\times\Gr_\bbR^\omega\to\Gr_\bbR.\]
The non-empty fibers of $m^k$ are paved by real affine space of real dimension divisible by
$n_X$.
\end{lemma}
\begin{proof}
The proof is similar to the case of real quaternionic groups 
in \cite[Lemma 4.11]{CMNO}.  
We give a description of the paving.  Let $\mu_*=\{\mu_1,...,\mu_k\}$ denote any list of minuscule coweights.  Let $\Gr^{\mu_*}_{\bbR} = \Gr^{\mu_1}_{\bbR} \tilde{\times} \dots \tilde{\times} \Gr^{\mu_k}_{\bbR}$. Consider the $k$-fold multiplication map
\[
m_k: \Gr^{\mu_1}_{\bbR} \tilde{\times} \dots \tilde{\times} \Gr^{\mu_k}_{\bbR} \rightarrow \Gr_\bbR
\]
We claim that $m_k$ is paved by real affine spaces of even dimension.  We work by induction.  For $k=1$ the claim is trivial.
Assume the claim holds for $k-1$ tuples of minuscule coweights.
We factor $m_k$ as 
\[
\Gr^{\mu_*}_{\bbR}\xrightarrow{m_{k-1}\tilde{\times}id} \Gr_\bbR\tilde{\times} \Gr^{\mu_k}_{\bbR}\xrightarrow{M_2} \Gr_\bbR.
\]
We build this inductively from the following diagram:  

$$\begin{tikzcd}
    \Gr^{\mu_*}_{\bbR}\arrow[d, "m_{k-1}\tilde{\times}id"] \arrow[r, "\pi_1"] & \Gr^{\mu_*\setminus \mu_{k}}_{\bbR} \arrow[d, "m_{k-1}"]\\
    \Gr_\bbR\tilde{\times} \Gr^{\mu_k}_{\bbR}\arrow[r, "\pi_1"] & \Gr_\bbR
\end{tikzcd}$$
Consider, for each $\lambda\in \Lambda_S$, the fiber $m_k^{-1}(t^\lambda)$.  We have that the fiber sits in the above diagram as follows:
\[\begin{tikzcd}
    m_k^{-1}(t^\lambda) \arrow[r, hook] \arrow[d, "m_{k-1}\tilde{\times}id"] &\Gr^{\mu_*}_{\bbR}\arrow[d, "m_{k-1}\tilde{\times}id"] \arrow[r, "\pi_1"] & \Gr^{\mu_*\setminus\mu_{k}}_{\bbR} \arrow[d, "m_{k-1}"]\\
    m_2^{-1}(t^\lambda) \arrow[r, hook] &\Gr_\bbR\tilde{\times} \Gr^{\mu_k}_{\bbR}\arrow[r, "\pi_1"] & \Gr_\bbR
\end{tikzcd}\]
The compositions along the rows are closed embeddings, hence we have a Cartesian square
\[
\begin{tikzcd}
    m_k^{-1}(t^\lambda) \arrow[r, "\sim"] \arrow[d, "m_{k-1}\tilde{\times}id"] & \pi_1(m_k^{-1}(t^\lambda)) \arrow[r, hook] \arrow[d, "m_{k-1}"]& \Gr^{\mu_*\setminus\mu_{k}}_{\bbR} \arrow[d, "m_{k-1}"]\\
     m_2^{-1}(t^\lambda) \arrow[r, "\sim"] &\pi_1(m_2^{-1}(t^\lambda)) \arrow[r, hook] & \Gr_\bbR
\end{tikzcd}
\]
Since $m_{k-1}: \Gr^{\mu_*\setminus\mu_{k}}_{\bbR} \rightarrow \Gr_\bbR$ has an affine paving, then so does $m_{k-1}\tilde{\times} id: m_k^{-1}(t^\lambda) \rightarrow m_2^{-1}(t^\lambda)$.  

We now need to show two things:  that $\pi_1(m_2^{-1}(t^\lambda))$ is paved by real affine spaces of even dimension, and that $m_{k-1}$ is a trivial fibration over it.  For the first, we observe that, by equivariance,
\[
m_2^{-1}(t^\lambda) = t^\lambda m_2^{-1}(e) = t^\lambda(\Gr^{-\mu_k}_{\bbR}\tilde{\times}\Gr^{\mu_k}_{\bbR}).
\]
Hence,
\[
\pi_1(m_2^{-1}(t^\lambda)) = t^\lambda \Gr^{-\mu_k}_{\bbR} = t^\lambda \Gr^{w_0(\mu_k)}_{\bbR}.
\]
where $w_0$ is the longest Weyl group element.  Since $\mu_k$ is minuscule, then so is $\nu = w_0(\mu_k)$.  Hence $m_2^{-1}(t^\lambda)\cong \Gr^{\nu}_{\bbR}$ is a real partial flag variety, by classification either $\mathrm S^{2n-2}$ or $\bbO P^2$.  The real Bruhat decomposition gives the affine paving of either.  More explicitly, for $w\in W_X$ and $P_{\mu,\bbR}$ the stabilizer of $t^\mu$ we have that
\[
m_2^{-1}(t^\lambda) = \bigsqcup_{w\in W_X} P_{\mu,\bbR}t^{w(\mu)}
\]
For every $w\in W_X$, there is a unipotent subgroup $N^w_{\mu,\bbR}$ of $w(P_{\mu,\bbR})$ such that $P_{\mu,\bbR}t^{w(\mu)} = N^w_{\mu,\bbR}t^{w(\mu)} \cong N^w_{\mu,\bbR}$ as real varieties.  We check the dimensions of $N^w_{\mu,\bbR}$:  for $PSO(2n-1,1)$ acting on $\mathrm S^{2n-2}$, the only nontrivial $N^w_{\mu,\bbR}$ is, up to (possibly outer) automorphism, the group of translations of $\mathrm S^{2n-2}\backslash\{\on{pt}\} \cong \bbR^{2n-2}.$  For $PE_6(F_4)$ acting on $\bbO P^2$, we have copies $\bbR^8 \cong \bbO$ acting on subspaces $\bbO P^1$ by translations fixing a point, as well as $\bbR^{16} \cong \bbO \oplus \bbO$ acting by translations fixing the octonionic line at infinity.  Hence, we have that $N^w_{\mu,\bbR}$ is of real dimension $0$, $2n-2$ (for $PSO(2n-1,1)$), or 0, 8, 16 (for $PE_6(F_4)$).  

 Now, consider the preimage of $N^w_{\mu,\bbR}t^{w(\mu)}$ under $m_{k-1}^{-1}$.  By equivariance, we have the square
\[\begin{tikzcd}
    N^w_{\mu,\bbR}\tilde{\times} m_{k-1}^{-1}(t^{w(\mu)})\arrow[d, "m_{k-1}\tilde{\times}id"] \arrow[r, "\pi_1"] & m_{k-1}^{-1}(N^w_{\mu,\bbR}t^{w(\mu)}) \arrow[d, "m_{k-1}"]\\
    N^w_{\mu,\bbR}\tilde{\times} t^{w(\mu)}\arrow[r, "\pi_1"] & N^w_{\mu,\bbR}t^{w(\mu)}
\end{tikzcd}\]

Here, $\pi_1$ is an isomorphism on the lower level, hence also on the upper level.  We know that $N^w_{\mu,\bbR}t^{w(\mu)}$ is isomorphic to an affine space of even dimension.  By induction, we have that $m_{k-1}^{-1}(t^{w(\mu)})$ is also paved by affine spaces.  Hence, collecting all $N^w_{\mu,\bbR}t^{w(\mu)}$, we get that $m_k$ is paved by affine spaces.
\end{proof}

Let $D^b_{L^+G_\bbR}(\Gr_{\bbR})$ be the bounded dg derived category of 
$L^+G_\bbR$-equivariant constructible complexes on $\Gr_\bbR$.
Since $\dim_\bbR\Gr_\bbR^\lambda=2\langle\lambda,\rho_G\rangle=2n_X
\langle\lambda,\check\rho_X\rangle$ is even for all $\lambda$, there is a natural 
perverse t-structure on $D^b_{L^+G_\bbR}(\Gr_{\bbR})$.  We denote by 
$\on{Perv}_{L^+G_\bbR}(\Gr_\bbR)$ the full subcategory of $L^+G_\bbR$-equivariant 
perverse sheaves on $\Gr_\bbR$ and $\IC_\lambda,\lambda\in\Lambda_S^+$
the $\IC$-complex for the orbit closure of $\Gr_\bbR^\lambda$.

\begin{corollary}\label{n_X vanishing}
We have $\mathscr H^{i-n_X\langle\lambda,\rho^\vee_X\rangle}(\IC_\lambda)=0$
for $n_X\nmid i$
\end{corollary}
\begin{proof}
Since $\dim_\bbR\Gr_\bbR^\lambda=2\langle\lambda,\rho_G\rangle\stackrel{\eqref{equ of paring}}=
2n_X\langle\lambda,\rho^\vee_X\rangle$,
the shifted $\IC$-complex $\IC_\lambda[-n_X\langle\lambda,\rho^\vee_X\rangle]$ is a direct summand 
of $(m^\lambda)_!(\bC)$ and desired claim follows from Lemma \ref{paving}.
\end{proof}

For any $\lambda\in\Lambda_S^+$,
let $i_\lambda:\Gr_\bbR^\lambda\to\Gr_\bbR$
 and $s_\lambda:\{t^\lambda\}\to\Gr_\bbR^\lambda
$ be  the natural inclusions

\begin{definition}
A complex $\mF\in D^b_{L^+G_\bbR}(\Gr_\bbR)$
is called even if for all $\lambda\in\Lambda_S^+$, both the 
$L^+G_\bbR$-complexes $i_\lambda^*\mF$ and 
$i^!_\lambda\mF$ are isomorphic to direct sum of constant sheaves 
appearing in even degree.
\end{definition}

\begin{lemma}\label{faithfulness}
(1) $\IC_\lambda$ is even for all $\lambda\in\Lambda_S^+$
(2) For any $\mu\in\Lambda_S^+$, the map
$H^*_{T_\bbR}(\IC_\lambda)\to H^*_{T_\bbR}(s_\mu^*\IC_\lambda)$
is surjective.
(3) For any $\lambda,\mu\in\Lambda_S^+$, the natural map
\[\on{Ext}^*_{D^b_{L^+G_\bbR}(\Gr_\bbR)}(\IC_\lambda,\IC_\mu)\to\Hom^*_{H^*_{L^+G_\bbR}(\Gr_\bbR)}(H^*(\IC_\lambda),H^*(\IC_\mu)) \]
is an isomorphism of graded modules.
\end{lemma}
\begin{proof}
Using Lemma \ref{paving}, the same argument in \cite[Section 4.7]{CMNO} apply to the case of 
$G_\bbR=PSO(2n-1,1)$ or $PE_6(F_4)$.
\end{proof}

\section{Ext algebras}

The real geometric Satake equivalence in \cite{N}
and Lemma \ref{trivial} implies that there is a monoidal equivalence 
\[\on{Rep}(G_X^\vee)\is\on{Perv}_{L^+G_\bbR}(\Gr_\bbR), V\to \calF_V.\]
Consider the monoidal $\on{Rep}(G_X^\vee)$-action $\star$ on 
$D^b_{L^+G_\bbR}(\Gr_\bbR)$ given by the convolution product
$\mF\star V:=\mF\star\mF_V$.
Let $\mO(G_X^\vee)$ be the regular representation of $G_X^\vee$ ans 
consider the dg extension algebra
$\on{Ext}^*_{D^b_{L^+G_\bbR}(\Gr_\bbR)}(\IC_0,\IC_0\star\mO(G_X^\vee))$.

\begin{thm}\label{computation of Ext}
There is a $G_X^\vee$-equivariant isomorphism of graded algebras
\[\on{Ext}^*_{D^b_{L^+G_\bbR}(\Gr_\bbR)}(\IC_0,\IC_0\star\mO(G_X^\vee))\is\mO(
\fg_X^\vee[n_X]\times\fl_X^\wedge[2]//L_X^\wedge)\]
where $G_X^\vee$ acts on the factor $\fg_X^\vee[n_X]$ via the adjoint 
representation and trivially on $\fl_X^\wedge[2]//L^\wedge_X$.
\end{thm}
\begin{proof}
By Lemma \ref{faithfulness}, 
taking equivariant cohomology induces  a $G_X^\vee$-equivariant isomorphism of 
graded algebras
\[\on{Ext}^*_{D^b_{L^+G_\bbR}(\Gr_\bbR)}(\IC_0,\IC_0\star\mO(G_X^\vee))
\is(\Hom^*_{H^*_{T_\bbR}(\Gr_{\bbR})}(H^*_{T_\bbR}(\IC_0),H^*_{T_\bbR}(\IC_0\star\mO(G_X^\vee))^{\rW_K}\is\]
\[\is
(\Hom^*_{H^*_{T_\bbR}(\Gr_\bbR)}(\mO(\ft),\mO(G_X^\vee\times\ft)))^{\rW_K}
\is(\mO(G_X^\vee\times\ft)^{\on{Spec}(H_*^{T_\bbR}(\Gr_{\bbR}))})^{\rW_K}
\is\]
\[\is
(\mO(G_X^\vee\times\ft)^{(G_X^\vee\times\ft)^{e^T}})^{\rW_K}\]
where 
$\mO(G_X^\vee\times\ft)^{(G_X^\vee\times\ft)^{e^T}}\subset\mO(G_X^\vee\times\ft)$
 is the subspace  consisting of 
functions 
that are invariant 
(relative over $\ft$) with respect to the left  action of the group scheme 
$\on{Spec}(H_*^{T_\bbR}(\Gr_{\bbR}))\is(G^\vee_X\times\ft)^{e^T}$ on $G^\vee_X\times\ft$.
Consider the fiber product 
\[\fg_X^\vee\times_{\ft_X^\vee//\rW_X}\ft\] with respect to the 
Chevalley map $\fg_X^\vee\to \fg_X^\vee//G_X^\vee\is\ft_X^\vee//\rW_X$
and the map
\[ \ft\stackrel{p^T}\to \ft_X^\vee\to \ft_X^\vee//\rW_X\] 
where $p^T$ is the  map in~\eqref{projection}.
We have a natural map
\beq\label{torsor}
\nu:G_X^\vee\times\ft\to (\fg_X^\vee)^{reg}\times_{\ft_X^\vee//\rW_X}\ft,\ \  (g,t)\to (\on{Ad}_{g^{-1}}e^T(t),t)
\eeq
which
realizes $G_X^\vee\times\ft$  as a $(G_X^\vee\times\ft)^{e^T}$-torsor over $(\fg_X^\vee)^{reg}\times_{\ft_X^\vee//\rW_X}\ft$. It follows that 
there is an isomorphism of algebras
\[\on{Ext}^*_{D^b_{L^+G_\bbR}(\Gr_{\bbR})}(\IC_0,\IC_0\star\mO(G_X^\vee))\is(\mO(G_X^\vee\times\ft)^{(G_X^\vee\times\ft)^{e^T}})^{\rW_K}\is\mO((\fg_X^\vee)^{reg}\times_{\ft_X^\vee//\rW_X}\ft)^{\rW_K}\]
\[\is
\mO(\fg_X^\vee\times_{\ft_X^\vee//\rW_X}\ft//\rW_K).\]
Finally, we note that Lemma \ref{equ of points} implies that 
the map $\ft//\rW_K\is (\ft//\rW_M)^{\rW_X}\to \ft_X^\vee//\rW_X$ in the fiber product 
$\fg_X^\vee\times_{\ft_X^\vee//\rW_X}\ft//\rW_K$
 is given by the projection 
\[\ft//\rW_K\is \ft//\rW_M\times \fl_X^\wedge//L_X^\wedge\to \ft//\rW_M\]
and we conclude that 
\[\on{Ext}^*_{D^b_{L^+G_\bbR}(\Gr_{\bbR})}(\IC_0,\IC_0\star\mO(G_X^\vee))\is
\mO(\fg_X^\vee\times_{\ft_X^\vee//\rW_X}\ft//\rW_K)\is \mO(\fg_X^\vee\times\fl_X^\wedge//L_X^\wedge).\]

It remains to check that the isomorphism above is compatible with the desired gradings.
 By \cite[Theorem 8.5.1]{N}, 
 for any $\lambda\in\Lambda_S$
 and $\mF\in\on{Perv}(\Gr_{\bbR})$,
the compactly supported cohomology $H^*_c(S_{\bbR}^\lambda,\mF)$ 
along the real semi-infinite orbit $S^\lambda_{\bbR}$
is non-zero only in degree
$\langle\lambda,\rho_{G}\rangle$.
Note that, by~\eqref{equ of paring}, we have 
$\langle\lambda,\rho_{G}\rangle=n_X\langle\lambda,\check\rho_X\rangle$
where in the second paring we view $\lambda$ as a weight of $G_X^\vee$.
Thus the grading on $H^*(\Gr_{\bbR},\IC_\lambda)$ corresponds, under the 
geometric Satake equivalence for $G_\bbR$, to the grading on $V_\lambda\in\on{Rep}(G_X^\vee)$
given by cocharacter $n_X\check\rho_X$
and it follows that 
the grading on $H^*_{T_\bbR}(\Gr_{\bbR},\IC_0\star\mO(G_X^\vee))\is\mO(G_X^\vee\times\ft)$
is induced by the $\bG_m$ action on $G_X^\vee\times\ft$ 
given by $x(g,t)=(n_X\check\rho_X(x)g,x^{-2}t)$ (note that the generators of $\mO(\ft)$ are in degree $2$).
We claim that the map $\nu$ in~\eqref{torsor} is $\bG_m$-equivariant 
with respect to the above $\bG_m$-action on $G_X^\vee\times\ft$ 
and the $\bG_m$-action on $(\fg_X^\vee)^{reg}\times_{\ft_X^\vee//\rW_X}\ft$ given by
$x(v,t)=(x^{-n_X}v,x^{-2}t)$.
Indeed, the map $e^T:\ft\to\fb_X^\vee$ satisfies
\[\on{Ad}_{n_X\check\rho_X(x^{-1})}e^T(x^{-2}t)=\on{Ad}_{n_X\check\rho_X(x^{-1})}(p^T(x^{-2}t)+\mathbf e)=x^{-n_X}p^T(t)+x^{-n_X}\mathbf e=
x^{-n_X}e^T(t)\]
and hence 
\[\nu(x(g,t))=
\nu(n_X\check\rho_X(x)g,x^{-2}t)=(\on{Ad}_{g^{-1}}\on{Ad}_{n_X\check\rho_X(x^{-1})}e^T(x^{-2}t),x^{-2}t)=\]
\[=(x^{-n_X}\on{Ad}_{g^{-1}}e^T(t),x^{-2}t)=x(\on{Ad}_{g^{-1}}e^T(t),t)=x\nu(g,t).\]
Thus the pull-back along the map $\nu$ induces an isomorphism 
of graded algebras
\[(\mO(G_X^\vee\times\ft)^{(G_X^\vee\times\ft)^{e^T}})^{\rW_K}\is\mO(\fg_X^\vee[n_X]\times_{\ft_X^\vee[n_X]//\rW_X}\ft[2]//W_K)\is\mO(\fg_X^\vee[n_X]\times\fl_X^\wedge[2]//L_X^\wedge).\]
This finishes the proof of the theorem.

\end{proof}

\section{Kostka-Foulkes polynomials}
For any $\lambda,\mu\in\Lambda_S^+$,
let $K_{\lambda,\mu}(q)$ be the Kostka-Foulkes polynomial of type $\fg_X^\vee$ 
associated to the pair $(\lambda,-\lambda)$ and $(\mu,-\mu)$.
\begin{thm}\label{Kostka-Foulkes}
Let $\lambda,\mu\in\Lambda_S^+$. For any $x\in\Gr_\bbR^\mu$,
we have 
\[K_{\lambda,\mu}(q)=\sum_{i}\mathscr H_x^{-n_Xi-n_X\langle\lambda,\check\rho_X\rangle}(\IC_\lambda)q^i\]
\end{thm}
\begin{proof}
The proof is similar to the case of quaternionic linear group in \cite[Section 4.9]{CMNO}.
For any $t\in\ft$ we denote by 
$\kappa(t)$ the residue field of $t$.
The specialized cohomology 
\[H_t(\Gr_{\bbR},\mF):=H^*_{T_\bbR}(\Gr_{\bbR},\mF)\otimes_{\mO(\ft)}\kappa(t)\]
carries a canonical filtration
\[H^{\leq i}_t(\Gr_{\bbR},\mF):=\on{Im}(\sum_{j\leq i} H^j_{T_\bbR}(\Gr_{\bbR}\to H_t(\Gr_{\bbR},\mF))\]
Let us identify 
$H_t(\Gr_{\bbR},\mF)\is (H^*(\Gr_{\bbR},\mF)\otimes\mO(\ft))\otimes_{\mO(\ft)}\kappa(t)\is V$ via the 
canonical splitting in \cite[Theorem 3]{O} induced from the real MV filtrations.
As explained in the proof of Theorem \ref{computation of Ext},
the cohomological grading  on
$H^*(\Gr_{n,\bbH},\mF)$ corresponds to the grading on the representation $V$
given by 
the eigenvalues of $n_X\check\rho_X$. 
It follows that 
the filtration $H^{\leq i}_t(\Gr_{n,\bbH},\mF)$ 
corresponds to  the increasing filtration on $V$ given by
the eigenvalues of $n_X\check\rho_X$.
Since the natural map  
$H^*_{T_\bbR}(\Gr_\bbR,\mF)\to H^*_{T_\bbR}(s_\mu^*\mF)$
is surjective by Lemma \ref{faithfulness} (2),  
the same argument as in \cite[Theorem 4.22]{CMNO} implies 
that 
\[K_{\lambda,\mu}(q)=\sum_{i}\mathscr H^{-n_Xi-n_X\langle\lambda,\check\rho_X\rangle}s_\mu^*(\IC_\lambda)q^i=
\sum_{i}\mathscr H_x^{-n_Xi-n_X\langle\lambda,\check\rho_X\rangle}(\IC_\lambda)q^i.\]
The proof is completed.
\end{proof}
\section{Derived Satake equivalence for Lorentz groups}

\begin{lemma}\label{trivial}
Any $L^+G_\bbR$-equivariant local system 
on $\Gr_\bbR^\lambda$ is isomorphic to the trivial local system 
with the canonical $L^+G_\bbR$-equivariant structure.
\end{lemma}
\begin{proof}For a stratum $\Gr^\lambda_{\bbR}$, the stabilizer of $t^\lambda$ retracts onto the corresponding Levi subgroup $L^\lambda_{\bbR}$.  If the Levi subgroup is connected, $\Gr^\lambda_{\bbR}$ will inherit no nontrivial local systems. We classify the Levi subgroups that may show up in this manner.  For $PSO(1,2n-1)$, the only such Levi subgroup is $S(O(1,1)\times O(2n-2))/\{\pm 1\} \cong \bbR_{>0} \times SO(2n-2)$, which is connected.  For $PE_6(F_4))$, the Levi subgroups are isomorphic to $Spin(1,9)$ or $Spin(8)$, both of which are connected.
\end{proof}

Let $D^{G_X^\vee}_{}(\mO(\fg_X^\vee[n_X]\times\fl_X^\wedge[2]//L^\wedge_X))$
be the dg-category of $G_X^\vee$-equivariant dg-modules 
over the dg-algebra $\mO(\fg_X^\vee[n_X]\times\fl_X^\wedge[2]//L^\wedge_X)$ (equipped with trivial differential).
Let $D_{L^+G_\bbR}(\Gr_{\bbR})=\on{Ind}(D^b_{L^+G_\bbR}(\Gr_{\bbR}))$
be the ind-completion of $D^b_{L^+G_\bbR}(\Gr_{\bbR})$.
\begin{thm}\label{main}
There is an equivalence of dg-categories
\[D_{L^+G_\bbR}(\Gr_{\bbR})\is D^{G_X^\vee}_{}(\mO(\fg_X^\vee[n_X]\times\fl_X^\wedge[2]//L^\wedge_X))\]
extending the 
real geometric Satake equivalence 
$\on{Perv}_{L^+G_\bbR}(\Gr_\bbR)\is\on{Rep}(G_X^\vee)$.
\end{thm}
\begin{proof}
Lemma \ref{trivial} implies that $\IC_0$ is compact and generates $D_{L^+G_\bbR}(\Gr_{\bbR})$
under the action of $\on{Rep}(G_X^\vee)$
and it follows from Barr-Beck-Lurie theorem that 
$D_{L^+G_\bbR}(\Gr_{\bbR})$ is equivalent the dg category 
of $G_X^\vee$-equivariant dg modules over the dg-algebra 
$\on{RHom}_{\calC}(\IC_0,\IC_0\star\mO(G_X^\vee))$.
By \cite[Theorem 13.4]{CN}, the above dg-algebra is formal 
and the computation of Ext algebra $\on{Ext}^*_{\calC}(\IC_0,\IC_0\star\mO(G_X^\vee))$
in Theorem \ref{computation of Ext} implies that 
$D_{L^+G_\bbR}(\Gr_{\bbR})$ is equivalent to the dg category of 
$G_X^\vee$-equivariant dg modules over the dg-algebra $\mO(\fg_X^\vee[n_X]\times\fl_X^\wedge[2]//L^\wedge_X)$ (equipped with trivial differential). 
The proof is completed

\end{proof}

\end{document}